\documentclass[twoside, 12pt]{article}
\usepackage{amsmath,amssymb,amsthm,mathrsfs}
\usepackage[margin=2.15cm]{geometry} 
\usepackage{lipsum}
\usepackage{titlesec,hyperref}
\usepackage{fancyhdr}
\usepackage[numbers,sort&compress]{natbib}
\usepackage{color}

\fancypagestyle{plain}
{
\fancyhf{}

}

\titleformat{\subsection}{\it}{\thesubsection.\enspace}{1.5pt}{}
\titleformat{\subsubsection}{\it}{\thesubsubsection.\enspace}{1.5pt}{}

\newtheorem{theo}{Theorem}[section]
\newtheorem{lemm}[theo]{Lemma}

\newtheorem{prop}[theo]{Proposition}
\newtheorem{rema}{Remark}[section]

\numberwithin{equation}{section}

\allowdisplaybreaks

\def\th2{\frac{\theta}{2}}

\begin{document}
\title{ Long-time Behavior of Solution for the Compressible Nematic Liquid Crystal Flows in $\mathbb{R}^3$  \hspace{-4mm}}
\author{Jincheng Gao$^\dag$  \quad Qiang Tao $^\ddag$ \quad Zheng-an Yao $^\dag$\\[10pt]
\small {$^\dag $School of Mathematics and Computational Science, Sun Yat-Sen University,}\\
\small {510275, Guangzhou, P. R. China}\\[5pt]
\small {$^\ddag$ College of Mathematics and Computational Science, Shenzhen University,}\\
\small {518060, Shenzhen, P. R. China}\\[5pt]
}


\footnotetext{Email: \it gaojc1998@163.com(J.C.Gao), taoq060@126.com(Q.Tao), \it mcsyao@mail.sysu.edu.cn(Z.A.Yao).}

\date{}

\maketitle

\begin{abstract}
In this paper, we investigate the Cauchy problem for the compressible nematic liquid crystal flows in
three-dimensional whole space. First of all, we establish the time decay rates
for compressible nematic liquid crystal flows by the method of spectral analysis and energy estimates.
Furthermore, we enhance the convergence rates for the higher-order spatial derivatives of density,
velocity and director. Finally, the time decay rates of mixed space-time derivatives of solution
are also established.

\vspace*{5pt}
\vspace*{5pt}


\end{abstract}


\section{Introduction}
\quad In this paper, we investigate the motion of compressible nematic liquid crystal flows,
which are governed by the following simplified version of the Ericksen-Leslie equations
\begin{equation}\label{1.1}
\left\{
\begin{aligned}
&\rho_t+{\rm div}(\rho u)=0,\\
&(\rho u)_t+{\rm div}(\rho u \otimes u)-\mu \Delta u-(\mu+\nu)\nabla {\rm div} u+\nabla P(\rho)
  =-\gamma \nabla d \cdot \Delta d,\\
&d_t+u\cdot \nabla d=\theta(\Delta d+|\nabla d|^2 d),
\end{aligned}
\right.
\end{equation}
where $\rho, u, P(\rho)$ and $d$ stand for the density, velocity, pressure and macroscopic average
of the nematic liquid crystal orientation field respectively. The constants $\mu$ and $\nu$
are shear viscosity and the bulk viscosity coefficients of the fluid, respectively, that satisfy
the physical assumptions
\begin{equation}\label{1.2}
\mu >0, ~ 2\mu+3\nu \ge 0.
\end{equation}
The positive constants $\gamma$ and $\theta$ represent the competition between the kinetic energy and the
potential energy, and the microscopic elastic relaxation time for the molecular orientation field, respectively.
For the sake of simplicity, we set the constants $\gamma$, $\theta$ and $P'(1)$ to be $1$. The symbol $\otimes$
denotes the Kronecker tensor product such that $u\otimes u=(u_i u_j)_{1\le i, j \le 3}$. To complete
the system \eqref{1.1}, the initial data is given by
\begin{equation}\label{1.3}
\left.(\rho, u, d)(x,t)\right|_{t=0}=(\rho_0(x), u_0(x), d_0(x)).
\end{equation}
As the space variable tends to infinity, we assume
\begin{equation}\label{1.4}
\underset{|x|\rightarrow \infty}{\lim}(\rho_0-1, u_0, d_0-w_0)(x)=0,
\end{equation}
where $w_0$ is an unit constant vector.
The system is a coupling between the compressible Navier-Stokes equations and a transported heat flow of
harmonic maps into $S^2$. It is a macroscopic continuum description of the evolution for the liquid crystals
of nematic type under the influence of both the flow field $u$ and the macroscopic description of the
microscopic orientation configuration $d$ of rod-like liquid crystals. Generally speaking, the system
\eqref{1.1} can not be obtained any better results than the compressible Navier-Stokes equations.

The hydrodynamic theory of liquid crystals in the nematic case has been established by Ericksen \cite{Ericksen}
and Leslie \cite{Leslie} during the period of $1958$ through $1968$.
Since then, the mathematical theory is still progressing and the study of the full Ericksen-Leslie model
presents relevant mathematical difficulties. The pioneering work comes from Lin and his partners
\cite{{Hard-Kinderlehrer-Lin}, {Lin},{Lin-Liu1}, {Lin-Liu2}}.
For example, Lin and Liu \cite{Lin-Liu1} obtained
the global weak and smooth solutions for the Ginzburg-Landau approximation to relax the nonlinear constraint
$d \in S^2.$ They also discussed the uniqueness and some stability properties of the system. Later, the decay
rates for this approximate system are given by Wu \cite{Wuhao} for a bounded domain and  Dai \cite{{Dai1},{Dai2}}
for the Cauchy problem respectively. Recently, Liu and Zhang \cite{Liu-Zhang}, for the density dependent model,
obtained the global weak solutions in dimension three with the initial density $\rho_0 \in L^2$, which was improved
by Jiang and Tan \cite{Jiang-Tan} for the case $\rho_0 \in L^{\gamma}(\gamma > \frac{3}{2})$.
Under the constraint $d \in S^2$, Wen and Ding \cite{Wen-Ding} established the local existence for the
strong solution and
obtained the global solution under the assumptions of small energy and positive initial density.
Later, Hong \cite{Hong} and Lin, Lin and Wang \cite{Lin-Lin-Wang} showed independently the global existence of a weak solution in two-dimensional space. Recently, Wang \cite{Wang}
established a global well-posedness theory for rough initial data provided that $\|u_0\|_{{\rm BMO}^{-1}}+[d_0]_{BMO} \le \varepsilon_0$ for some $\varepsilon_0 >0.$  Under this condition, Du and Wang \cite{Du-Wang1} obtained arbitrary space-time regularity for the Koch and Tataru type solution $(u, d)$. As a corollary, they also got the decay rates.
Very Recently, Lin and Wang \cite{Lin-Wang} established the global existence of a weak solution for the initial-boundary value or the Cauchy problem by restricting the initial director field on the unit upper hemisphere.
For more results, readers can refer to \cite{{Huang-Wang},{Du-Wang2},{Li2},{Li-Wang1},{Li-Wang2},{Hao-Liu}}
and references therein.

Considering the compressible nematic liquid crystal flows \eqref{1.1}, Ding, Lin, Wang and Wen
\cite{Ding-Lin-Wang-Wen} have gained
both existence and uniqueness of global strong solution for the one dimensional space.
And this result about the classical solution
was improved by Ding, Wang and Wen \cite{Ding-Wang-Wen} by generalizing the fluids to be of vacuum.
For the case of multi-dimensional space, Jiang, Jiang and Wang \cite{Jiang-Jiang-Wang}
established the global existence of weak solutions to the initial-boundary problem with
large initial energy and without any smallness condition on the initial density and
velocity if some component of initial direction field is small. Recently, Lin, Lai and Wang
\cite{Lin-Lai-Wang} established the existence of global weak solutions in three-dimensional space,
provided the initial orientational director field $d_0$ lies in the hemisphere $S_2^+$.
Local existence of unique strong solution was proved if that the initial data $(\rho_0, u_0, d_0)$
was sufficiently regular and satisfied a natural compatibility
condition in a recent work \cite{Huang-Wang-Wen1}. Some blow-up criterions that were derived for the possible breakdown
of such local strong solution at finite time could be found in \cite{{Huang-Wang-Wen2},{Huang-Wang1}, {Gao-Tao-Yao}}.
The local existence and uniqueness of classical solution to \eqref{1.1} was established by Ma in \cite{Ma}.
On the other hand, Hu and Wu \cite{Hu-Wu} obtained the existence and uniqueness of global strong solution
in critical Besov spaces
provided that the initial data was close to an equilibrium state $(1,0, \hat{d})$ with a constant vector
$\hat{d}\in S^2$. For more results, the readers can refer to \cite{Lin-Wang1} that have introduced
some recent developments of analysis for hydrodynamic flow of nematic liquid crystal
flows and references therein.

If the director is an unit constant vector, then the compressible nematic liquid crystal flow \eqref{1.1}
becomes the compressible Navier-Stokes equations. The convergence rates
of solution for the compressible Navier-Stokes equations to the steady state has been investigated extensively
since the first global existence of small solutions in $H^3$(classical solutions) was improved by
Matsumura and Nishida \cite{Matsumura-Nishida1}.
For the small initial perturbation belongs to $H^3$ only,
Matsumura \cite{Matsumura} took  weighted energy method to show the optimal time decay rates
\begin{equation*}
\|\nabla^k(\rho-1,u)(t)\|_{L^2}\lesssim (1+t)^{-\frac{1}{2}}
\end{equation*}
for $k=1,2$ and
\begin{equation*}
\|(\rho-1,u)(t)\|_{L^\infty}\lesssim (1+t)^{-\frac{3}{4}}.
\end{equation*}
Furthermore, for the initial perturbation small in $L^1\cap H^3$,
Matsumura and Nishida \cite{Matsumura-Nishida2} obtained
\begin{equation*}
\|(\rho-1,u)(t)\|_{L^2}\lesssim (1+t)^{-\frac{3}{4}},
\end{equation*}
and for the small initial perturbation belongs to $H^m\cap W^{m,1}$, with $m\ge4$, Ponce \cite{Ponce}
proved the optimal $L^q$ decay rates
\begin{equation*}
\|\nabla^k(\rho-1,u)(t)\|_{L^q}\lesssim (1+t)^{-\frac{3}{2}(1-\frac{1}{q})-\frac{k}{2}}
\end{equation*}
for $2\le q \le \infty$ and $0 \le k \le 2$. With the help of the study of Green function,
the optimal $L^q(1\le q \le \infty)$ decay rates were also obtained
\cite{{Hoff-Zumbrum1},{Hoff-Zumbrum2},{Liu-Wang}}
for the small initial
perturbation to $H^m \cap L^1$ with $m \ge4.$ These results were extended to the exterior problem
\cite{{Kobayashi-Shibata},{Kobayashi}}
or the half space problem \cite{{Kagei-Kobayashi1},{Kagei-Kobayashi2}}
or with an external potential force \cite{Duan-Yang},
but without the smallness of
$L^1-$norm of the initial perturbation.
While based on a differential inequality, Deckelnick in \cite{{Deckelnick1},{Deckelnick2}}
obtained a slower (than the optimal) decay rate for the problem in unbounded domains
with external force through the pure energy method.
Recently, Guo and Wang \cite{Guo-Wang} developed a general energy method for proving the optimal time
decay rates of the solutions in the whole space as
\begin{equation}\label{Guo}
\|\nabla^l(\rho-1, u)(t)\|_{H^{N-l}}^2\lesssim (1+t)^{-(l+s)}
\end{equation}
for $0 \le l \le N-1$ by assuming the initial data
$\|(\rho_0-1,u_0)\|_{\dot{H}^{-s}}(s\in [0, \frac{3}{2}))$ is finite additionally.
This result was improved by Wang \cite{Wang-Yan-Jin} to establish the global existence of solution
by assuming the smallness of initial data of $H^3$ rather than $H^{\left[\frac{N}{2}\right]+2}$ norm.

In this paper, we establish the global solution by the energy method \cite{Guo-Wang} under the
assumption of smallness of $\|(\rho_0-1, u_0, \nabla d_0)\|_{H^3}$.
The difficulty for the system \eqref{1.1} is to deal with the nonlinear term $\nabla d \cdot \Delta d$
and the supercritical nonlinear term $|\nabla d|^2 d$. Since this proof is standard, we only sketch it as
an appendix for brevity. Furthermore, by assuming that $\|d_0-w_0\|_{L^2}$ and
$\|(\rho_0-1, u_0,  d_0-w_0)\|_{L^1} $ are finite additionally,  we establish the time decay rates
for the compressible nematic liquid crystal flows by the method of spectral analysis and energy estimates.
In order to improve the time decay rates of higher-order spatial derivatives of solution,
we take the strategy of induction that is promised by the Fourier splitting method.
Finally, we also study the decay rates for the mixed space-time derivatives of density, velocity and director.

\textbf{Notation:} In this paper, we use $H^s(\mathbb{R}^3)( s\in \mathbb{R})$ to denote the usual Sobolev spaces
with norm $\|\cdot\|_{H^s}$ and $L^p(\mathbb{R}^3)(1\le p \le \infty)$ to denote the usual $L^p$ spaces with norm
$\| \cdot \|_{L^p}$. The symbol $\nabla^l $ with an integer $l \ge 0$ stands for the usual any spatial derivatives
of order $l$. When $l$ is not an integer, $\nabla^l $ stands for $\Lambda^l$ defined by
$\Lambda^l f :=\mathscr{F}^{-1}(|\xi|^l \mathscr{F}f)$, where $\mathscr{F}$ is the usual Fourier transform operator
and $\mathscr{F}^{-1}$ its inverse.
We also denote $\mathscr{F}(f):=\hat{f}$.
The notation $a \lesssim b$ means that $a \le C b$ for a universal constant $C>0$ independent of
time $t$.
For the sake of simplicity, we write $\|(A, B)\|_X=\|A\|_X+\|B\|_X$ and $\int f dx:=\int _{\mathbb{R}^3} f dx.$

Now, we state our first result concerning the global existence of solution to the compressible nematic liquid
crystal flows \eqref{1.1}-\eqref{1.3} as follows.
\begin{theo}\label{Global-existence}
Assume that $(\rho_0-1, u_0, \nabla d_0) \in H^N$ for any integer $N \ge 3$ and there
exists a constant $\delta >0$ such that
\begin{equation}\label{1.6}
\|(\rho_0-1, u_0, \nabla d_0)\|_{H^3} \le \delta,
\end{equation}
then the problem \eqref{1.1}-\eqref{1.3} admits a unique global solution $(\rho, u, d)$ satisfying
for all $t \ge 0$,
\begin{equation}\label{1.7}
\begin{aligned}
\|(\rho-1, u, \nabla d)\|_{H^N}^2+
\int_0^t (\|\nabla \rho\|_{H^{N-1}}^2+\|(\nabla u, \nabla^2 d)\|_{H^N}^2 )d\tau
\le C \|(\rho_0-1, u_0, \nabla d_0)\|_{H^N}^2.
\end{aligned}
\end{equation}
\end{theo}

After having the global existence of solution for the compressible nematic liquid crystal flows
\eqref{1.1}-\eqref{1.3} at hand, we hope to investigate the long-time behavior of solution.
Hence, we establish the following time decay rates.

\begin{theo}\label{Decay1}
Under all the assumptions of Theorem \ref{Global-existence}, assuming the initial data $\|d_0-w_0\|_{L^2}$ and
$\|(\rho_0-1, u_0, d_0-w_0)\|_{L^1}$ are finite additionally, then the global solution $(\rho, u, d)$ of problem
\eqref{1.1}-\eqref{1.4} satisfies
\begin{equation}\label{1.8}
\begin{aligned}
&\|\nabla^k (\rho-1)(t)\|_{H^{N-k}}
 +\|\nabla^k u (t)\|_{H^{N-k}} \le C (1+t)^{-\frac{3+2k}{4}},\\
&\|\nabla^l (d-w_0)(t)\|_{L^2} \le C (1+t)^{-\frac{3+2l}{4}},
\end{aligned}
\end{equation}
where $k=0,1,...,N-1$ and $l=0,1,2,...,N+1$.
\end{theo}

\begin{rema}\label{remark1.4}
For any $2\le p \le 6$, by virtue of Theorem \ref{Decay1} and Sobolev interpolation inequality,
we also obtain the following time decay rates:
\begin{equation*}
\begin{aligned}
&\|\nabla^k (\rho-1)(t)\|_{L^p}+\|\nabla^k u(t)\|_{L^p}
 \le C (1+t)^{-\frac{3}{2}\left(1-\frac{1}{p}\right)-\frac{k}{2}},\\
&\|\nabla^l (d-w_0)(t)\|_{L^p} \le C (1+t)^{-\frac{3}{2}\left(1-\frac{1}{p}\right)-\frac{l}{2}},
\end{aligned}
\end{equation*}
where $k=0,1,...,N-2$ and $l=0,1,2,...,N$. Furthermore, it is easy to get the
convergence rates,
\begin{equation*}
\begin{aligned}
&\|\nabla^k (\rho-1)(t)\|_{L^\infty}+\|\nabla^k u(t)\|_{L^\infty}
 \le C (1+t)^{-\frac{3+k}{2}},\\
&\|\nabla^l (d-w_0)(t)\|_{L^\infty} \le C (1+t)^{-\frac{3+l}{2}},
\end{aligned}
\end{equation*}
where $k=0,1,...,N-3$ and $l=0,1,2...,N-1$.
\end{rema}

\begin{rema}
Generally speaking, the application of spectral analysis and energy estimates only helps the higher-order
spatial derivatives of solution obtain the same decay rates as the first-order spatial derivatives of solution when we deal with the compressible Navier-Stokes equations. In order to improve the convergence
rates for the higher-order spatial derivatives of solution, Guo and Wang \cite{Guo-Wang} developed
a general energy method when the initial data belongs to some negative Sobolev space
additionally. In this paper, we apply the Fourier splitting method by Schonbek \cite{Schonbek1} or \cite{Schonbek2} to improve the convergence rates for the higher-order spatial derivatives of solution.
The advantage of our results \eqref{1.8} is to verify that \eqref{Guo} holds on for the case $s=\frac{3}{2}$. Furthermore, the time decay rates in \eqref{1.8}
\begin{equation*}
\|\nabla^{N+1} (d-w_0)(t)\|_{L^2} \le C (1+t)^{-\frac{5+2N}{4}},
\end{equation*}
is completely new.
\end{rema}

Finally, we turn to study the last results concerning the time decay rates for the mixed space-time derivatives
to the compressible nematic liquid crystal flows \eqref{1.1}-\eqref{1.4}.
\begin{theo}\label{Decay2}
Under all the assumptions of Theorem \ref{Decay1}, then the global solution $(\rho, u, d)$ of problem
\eqref{1.1}-\eqref{1.4} has the time decay rates:
\begin{equation}\label{1.9}
\begin{aligned}
&\|\nabla^k \rho_t(t)\|_{H^{N-1-k}}+\|\nabla^k u_t(t)\|_{L^2}
\le C(1+t)^{-\frac{5+2k}{4}},\\
&\|\nabla^l d_t(t)\|_{L^2}\le C(1+t)^{-\frac{7+2l}{4}}\\
\end{aligned}
\end{equation}
for $k=0,1,...,N-2$ and $l=0,1,...,N-1$.
\end{theo}

This paper is organized as follows. Since the proof of global existence of solution is standard, we only
sketch it in section $4$ as an appendix.  In section $2$, we establish the time convergence rates for
the density, velocity and direction field to
the compressible nematic liquid crystal flows \eqref{1.1}-\eqref{1.4}. More precisely, the decay rates
are built by the method of spectral analysis, energy estimates and Fourier splitting method.
In section $3$, we also study the time decay rates for mixed space-time derivatives of density, velocity
and director.

\section{Proof of Theorem \ref{Decay1}}

\quad In this section, we investigate the time decay rates for the compressible nematic liquid crystal flows
\eqref{1.1}-\eqref{1.4} when the initial data belongs to $L^1$ space additionally.
First of all, we derive the decay rates for the linearized compressible nematic liquid crystal flows
that are a coupling of linearized compressible Navier-Stokes equations and heat equations.
Secondly, we establish the decay rates for the compressible nematic liquid crystal flows
\eqref{1.1}-\eqref{1.4} by the method of spectral analysis and energy estimates. Furthermore, we improve the decay rates
for higher-order spatial derivatives of density, velocity and direction field.

Denoting $\varrho=\rho-1$ and $n=d-w_0$, then we rewrite \eqref{1.1} in the perturbation form as
\begin{equation}\label{2.1}
\left\{
\begin{aligned}
& \varrho_t+{\rm div}u=S_1,\\
& u_t-\mu \Delta u-(\mu+\nu)\nabla {\rm div}u+\nabla \varrho=S_2,\\
& n_t-\Delta n=S_3.
\end{aligned}
\right.
\end{equation}
Here $S_i(i=1,2,3)$ are defined as
\begin{equation}\label{2.2}
\left\{
\begin{aligned}
& S_1=-\varrho{\rm div}u-u \cdot \nabla \varrho,\\
& S_2=-u \cdot \nabla u-h(\varrho)[\mu\Delta u+(\mu+\nu)\nabla {\rm div}u]
-f(\varrho)\nabla \varrho -g(\varrho)\nabla n \cdot \Delta n,\\
& S_3=-u\cdot \nabla n+|\nabla n|^2 (n+w_0),
\end{aligned}
\right.
\end{equation}
where the three nonlinear functions of $\varrho$ are defined by
\begin{equation}\label{2.3}
h(\varrho):=\frac{\varrho}{\varrho+1},\quad  f(\varrho):=\frac{P'(\varrho+1)}{\varrho+1}-1 \quad {\text {and}} \quad
g(\varrho):=\frac{1}{\varrho+1}.
\end{equation}
The associated initial condition is given by
\begin{equation}\label{2.4}
\left.(\varrho, u, n)\right|_{t=0}=(\varrho_0, u_0, n_0).
\end{equation}
By virtue of \eqref{1.7} and Sobolev inequality, it is easy to get
\begin{equation*}
\frac{1}{2}\le \varrho+1 \le \frac{3}{2}.
\end{equation*}
Hence, we immediately have
\begin{equation}\label{2.5}
|h(\varrho)|, |f(\varrho)| \le C|\varrho| \quad  {\text{and}} \quad |g^{k-1}(\varrho)|, |h^k(\varrho)|, |f^k(\varrho)| \le C \quad {\text{for any}} \quad k \ge 1,
\end{equation}
which we will use frequently to derive the a priori estimates for the time decay rates.

Now, we state the classical Sobolev interpolation of the Gagliardo-Nirenberg inequality, refer to \cite{Nirenberg}.
\begin{lemm}\label{lemma2.1}
Let $0 \le m, \alpha \le l$ and the function $f\in C_0^\infty(\mathbb{R}^3)$, then we have
\begin{equation}\label{GN}
\|\nabla^\alpha f\|_{L^p} \lesssim \|\nabla^m f \|_{L^2}^{1-\theta} \|\nabla^l f \|_{L^2}^\theta,
\end{equation}
where $0\le \theta \le 1$ and $\alpha$ satisfy
\begin{equation*}
\frac{1}{p}-\frac{\alpha}{3}=\left(\frac{1}{2}-\frac{m}{3}\right)(1-\theta)
+\left(\frac{1}{2}-\frac{l}{3}\right)\theta.
\end{equation*}
\end{lemm}

On the other hand, the following lemma is very useful when we deal with the nonlinear function of $\varrho$,
refer to \cite{Wang-Yan-Jin}.
\begin{lemm}\label{lemma2.2}
Assume that $\|\varrho\|_{L^\infty}\le 1$. Let $g(\varrho)$ be a smooth function of $\varrho$ with
bounded derivatives of any order, then for any integer $m \ge 1$ we have
\begin{equation}\label{2.7}
\|\nabla^m (g(\varrho))\|_{L^\infty} \lesssim \|\nabla^m \varrho\|_{L^\infty}.
\end{equation}
\end{lemm}

\subsection{Decay rates for the nonlinear systems}

\quad First of all, let us to consider the following linearized compressible nematic liquid crystal systems
\begin{equation}\label{2.8}
\left\{
\begin{aligned}
& \varrho_t+{\rm div}u=0,\\
& u_t-\mu \Delta u-(\mu+\nu)\nabla {\rm div}u+\nabla \varrho=0,\\
& n_t-\Delta n=0,
\end{aligned}
\right.
\end{equation}
with the initial data
\begin{equation}\label{2.9}
\left.(\varrho, u, n)\right|_{t=0}=(\varrho_0, u_0, n_0).
\end{equation}
Obviously, the solution $(\varrho, u, n)$ of the linear problem \eqref{2.8}-\eqref{2.9} can be expressed as
\begin{equation}\label{2.10}
(\varrho, u, n)^{tr}=G(t)*(\varrho_0, u_0, n_0)^{tr}, t \ge 0.
\end{equation}
Here $G(t):=G(x,t)$ is the Green matrix for the system \eqref{2.8}
and the exact expression of the Fourier transform $\hat{G}(\xi, t)$ of Green function $G(x,t)$ as
\begin{equation*}
\hat{G}(\xi, t)=
\left[
\begin{array}{ccc}
\frac{\lambda_+ e^{\lambda_- t}-\lambda_- e^{\lambda_+ t}}{\lambda_+ -\lambda_-}
& \frac{-i \xi^t (e^{\lambda_+ t}-e^{\lambda_- t})}{\lambda_+ -\lambda_-}
& 0\\
 \frac{-i \xi (e^{\lambda_+ t}-e^{\lambda_- t})}{\lambda_+ -\lambda_-}
& \frac{\lambda_+ e^{\lambda_+ t}-\lambda_- e^{\lambda_- t}}{\lambda_+ -\lambda_-} \frac{\xi \xi^t}{|\xi|^2}
  +e^{\lambda_0 t}\left(I_{3\times 3}- \frac{\xi \xi^t}{|\xi|^2}\right)
&0 \\
0
&0
&e^{\lambda_1 t}I_{3\times 3}\\
\end{array}
\right]
\end{equation*}
where
\begin{equation*}
\begin{aligned}
&\lambda_0 =-\mu |\xi|^2,\quad \lambda_1 =-|\xi|^2,\\
&\lambda_+ =-\left(\mu+\frac{1}{2}\nu\right)|\xi|^2+ i \sqrt{|\xi|^2-\left(\mu+\frac{1}{2}\nu\right)^2|\xi|^4},\\
&\lambda_- =-\left(\mu+\frac{1}{2}\nu\right)|\xi|^2- i \sqrt{|\xi|^2-\left(\mu+\frac{1}{2}\nu\right)^2|\xi|^4}.\\
\end{aligned}
\end{equation*}
Since the systems \eqref{2.8} is a decoupled system of the classical linearized Navier-Stokes equations
and heat equations,
the representation of Green function $\hat{G}(\xi, t)$ is easy to verify. Furthermore, we have the following
decay rates for the linearized systems \eqref{2.8}-\eqref{2.9}, refer to \cite{{Hu-Wu}}.

\begin{prop}\label{proposition2.3}
Let $N\ge 3$ be an integer. Assume that $(\varrho, u, n)$ is the solution of the linearized compressible nematic liquid
crystal system \eqref{2.8}-\eqref{2.9} with the initial data $(\varrho_0, u_0, n_0)\in H^N \cap L^1$, then
\begin{equation*}
\begin{aligned}
&\|\nabla^k \varrho\|_{L^2}^2\le C\left(\|(\varrho_0, u_0)\|_{L^1}^2+\|\nabla^k(\varrho_0, u_0)\|_{L^2}^2\right)
 (1+t)^{-\frac{3}{2}-k},\\
&\|\nabla^k u\|_{L^2}^2\le C\left(\|(\varrho_0, u_0)\|_{L^1}^2+\|\nabla^k(\varrho_0, u_0)\|_{L^2}^2\right)
 (1+t)^{-\frac{3}{2}-k},\\
&\|\nabla^k n\|_{L^2}^2\le C\left(\|n_0\|_{L^1}^2+\|\nabla^k n_0\|_{L^2}^2\right)
 (1+t)^{-\frac{3}{2}-k}\\
\end{aligned}
\end{equation*}
for $0\le k \le N$.
\end{prop}

The following estimates are essential for us to establish the time decay rates by the method of Green function.
Since it is easy to derive, then we only state the results here for the brevity. To be precise, we have
\begin{equation}\label{2.11}
\begin{aligned}
\|(S_1, S_2, S_3)\|_{L^1}
&\le (\|\varrho\|_{L^2}+\|u\|_{L^2}+\|\nabla n\|_{L^2})(\|\nabla \varrho\|_{L^2}+\|\nabla u\|_{H^1}+\|\nabla n\|_{H^1})\\
&\lesssim \delta (\|\nabla \varrho\|_{L^2}+\|\nabla u\|_{H^1}+\|\nabla n\|_{H^1}),\\
\|(S_1, S_2, S_3)\|_{L^2}
&\le (\|\varrho\|_{H^1}+\|u\|_{H^1}+\|\nabla n\|_{H^1})(\|\nabla^2 \varrho\|_{L^2}+\|\nabla^2 u\|_{H^1}+\|\nabla^2 n\|_{H^1})\\
&\lesssim \delta (\|\nabla^2 \varrho\|_{L^2}+\|\nabla^2 u\|_{H^1}+\|\nabla^2 n\|_{H^1}),\\
\|\nabla (S_1, S_2, S_3)\|_{L^2}
&\le (\|\varrho\|_{H^2}+\|u\|_{H^2}+\|\nabla n\|_{H^2})(\|\nabla^2 \varrho\|_{H^1}+\|\nabla^2 u\|_{H^1}+\|\nabla n\|_{H^2})\\
&\lesssim \delta (\|\nabla^2 \varrho\|_{H^1}+\|\nabla^2 u\|_{H^1}+\|\nabla n\|_{H^2}).\\
\end{aligned}
\end{equation}

Before studying the time decay rates, the following energy estimates will be used to
guarantee the first-order derivatives of velocity and director enjoying the same convergence rates.

\begin{lemm}\label{lemma2.4}
Under the assumption \eqref{1.6}, then we have for any integer $k=0,1, 2, ..., N,$
\begin{equation}\label{241}
\begin{aligned}
\frac{d}{dt}\int |\nabla^k n|^2 dx+\int |\nabla^{k+1} n|^2 dx \lesssim \delta \|\nabla^{k+1} u\|_{L^2}^2.
\end{aligned}
\end{equation}
\end{lemm}
\begin{proof}
Taking $k-$th spatial derivatives to $\eqref{2.1}_3$, multiplying the resulting equations by $\nabla^k n$
and integrating over $\mathbb{R}^3$(by part), it arrives at
\begin{equation}\label{242}
\begin{aligned}
&\frac{1}{2}\frac{d}{dt}\int |\nabla^k n|^2 dx+\int |\nabla^{k+1} n|^2 dx\\
&=\int \nabla^k \left[- u \cdot \nabla n+|\nabla n|^2 (n+w_0)\right]\cdot \nabla^k n \ dx
=I_1+I_2.
\end{aligned}
\end{equation}
Since the inequality \eqref{241} is easy to verify for the case $k=0$, then we only
verify the case $k \ge 1$.
Applying Leibnitz formula and Holder inequality, it is easy to deduce
\begin{equation}\label{243}
\begin{aligned}
I_1
&=\int\sum_{l=0}^{k-1} C_{k-1}^l \nabla^{l} u \nabla^{k-l}n\ \nabla^{k+1} n \ dx\\
&\lesssim \sum_{l=0}^{k-1} \|\nabla^{l} u\|_{L^3} \|\nabla^{k-l} n\|_{L^6} \|\nabla^{k+1} n \|_{L^2}.
\end{aligned}
\end{equation}
For the case $0 \le l \le \left[ \frac{k-1}{2}\right]$, by virtue of the interpolation inequality \eqref{GN}
and Young inequality, we have
\begin{equation}\label{244}
\begin{aligned}
&\quad \|\nabla^{l} u\|_{L^3} \|\nabla^{k-l}n\|_{L^6}\|\nabla^{k+1} n\|_{L^2}\\
&\lesssim \|\nabla^\alpha u\|_{L^2}^{1-\frac{l}{k}} \|\nabla^{k+1} u\|_{L^2}^{\frac{l}{k}}
          \|\nabla n\|_{L^2}^{\frac{l}{k}} \|\nabla^{k+1} n\|_{L^2}^{1-\frac{l}{k}}\|\nabla^{k+1} n\|_{L^2}\\
&\lesssim \delta (\|\nabla^{k+1}u\|_{L^2}^2+\|\nabla^{k+1} n\|_{L^2}^2),
\end{aligned}
\end{equation}
where $\alpha$ is defined
\begin{equation*}
\alpha=1-\frac{k}{2(k-l)}\in \left[0, \frac{1}{2}\right].
\end{equation*}
Similarly, for the case $\left[\frac{k-1}{2}\right]+1 \le l \le k-1,$ it follows that
\begin{equation}\label{245}
\begin{aligned}
&\quad \|\nabla^{l} u\|_{L^3} \|\nabla^{k-l}n\|_{L^6}\|\nabla^{k+1} n\|_{L^2}\\
&\lesssim \|u\|_{L^2}^{1-\frac{l+\frac{1}{2}}{k+1}} \|\nabla^{k+1} u\|_{L^2}^{\frac{l+\frac{1}{2}}{k+1}}
          \|\nabla^\alpha n\|_{L^2}^{\frac{l+\frac{1}{2}}{k+1}} \|\nabla^{k+1} n\|_{L^2}^{1-\frac{l+\frac{1}{2}}{k+1}}
          \|\nabla^{k+1} n\|_{L^2}\\
&\lesssim \delta (\|\nabla^{k+1}u\|_{L^2}^2+\|\nabla^{k+1} n\|_{L^2}^2),
\end{aligned}
\end{equation}
where $\alpha$ is defined by
\begin{equation*}
\alpha=\frac{k+1}{2l+1} \in \left(\frac{1}{2},1\right].
\end{equation*}
Indeed, the smallness of $\|\nabla^\alpha n\|(0 < \alpha \le 1)$ is guaranteed by the boundedness of
$\|n\|_{L^2}$ and the smallness of $\|\nabla n\|_{L^2}$.
Plugging \eqref{244} and \eqref{245} into \eqref{243}, we obtain
\begin{equation}\label{246}
I_1 \lesssim \delta (\|\nabla^{k+1} u\|_{L^2}^2+\|\nabla^{k+1} n\|_{L^2}^2).
\end{equation}
By virtue of Leibnitz formula and Holder inequality, it arrives at
\begin{equation}\label{247}
\begin{aligned}
I_2
& =-\int \nabla^{k-1}(|\nabla n|^2 (n+w_0))\nabla^{k+1}  n \ dx\\
& =-\int \sum_{l=0}^{k-1} C_{k-1}^l \nabla^l(|\nabla n|^2) \nabla^{k-1-l} (n+w_0) \nabla^{k+1} n \ dx\\
& =-\int \sum_{l=0}^{k-1} \sum_{m=0}^l C_{k-1}^l  C_l^m
     \nabla^{m+1}n \nabla^{l+1-m}n \nabla^{k-1-l} (n+w_0) \nabla^{k+1} n \ dx\\
& =-\int |\nabla n|^2 \nabla^{k-1} n \nabla^{k+1} n \ dx
   -\int \sum_{m=0}^{k-1} C_{k-1}^m  \nabla^{m+1}n\nabla^{k-m}n  (n+w_0)\nabla^{k+1}n \ dx\\
&\quad -\int \sum_{l=1}^{k-2} \sum_{m=0}^{l} C_{k-1}^l  C_l^m
    \nabla^{m+1} n \nabla^{l+1-m}n \nabla^{k-1-l}(n+w_0) \nabla^{k+1}n \ dx\\
&=I_{21}+I_{22}+I_{23}.
\end{aligned}
\end{equation}
With the help of Holder inequality and interpolation inequality \eqref{GN}, we deduce directly
\begin{equation}\label{248}
\begin{aligned}
I_{21}
&\le \|\nabla n\|_{L^6}\|\nabla n\|_{L^6}\|\nabla^{k-1} n\|_{L^6}\|\nabla^{k+1} n\|_{L^2}\\
&\lesssim \|\nabla^2 n\|_{L^2} \|\nabla n\|_{L^2}^{1-\frac{1}{k}}\|\nabla^{k+1} n\|_{L^2}^{\frac{1}{k}}
          \|\nabla n\|_{L^2}^{\frac{1}{k}}\|\nabla^{k+1}n\|_{L^2}^{1-\frac{1}{k}}
          \|\nabla^{k+1} n\|_{L^2}\\
&\lesssim \|\nabla^2 n\|_{L^2}\|\nabla n\|_{L^2}\|\nabla^{k+1} n\|_{L^2}^2\\
&\lesssim \delta \|\nabla^{k+1} n\|_{L^2}^2.
\end{aligned}
\end{equation}
In order to estimate the term $I_{22}$, using Holder inequality and \eqref{GN},
we obtain, for the case $0 \le m \le \left[\frac{k-1}{2}\right]$, that
\begin{equation}\label{249}
\begin{aligned}
&\quad \|\nabla^{m+1} n\|_{L^3}\|\nabla^{k-m}n\|_{L^6}\|\nabla^{k+1}n\|_{L^2}\\
&\lesssim \|\nabla^\alpha n\|_{L^2}^{1-\frac{m}{k}}\|\nabla^{k+1} n\|_{L^2}^{\frac{m}{k}}
         \|\nabla n \|_{L^2}^{\frac{m}{k}}\|\nabla^{k+1} n\|_{L^2}^{1-\frac{m}{k}}
         \|\nabla^{k+1} n\|_{L^2}\\
&\lesssim \delta \|\nabla^{k+1} n\|_{L^2}^2,
\end{aligned}
\end{equation}
where $\alpha$ is defined by
\begin{equation*}
\alpha=1+\frac{k}{2(k-m)} \in \left[\frac{3}{2},2\right].
\end{equation*}
Similarly, for the case $\left[\frac{k-1}{2}\right]+1 \le m \le k-1$, it is easy to deduce
\begin{equation}\label{2410}
\begin{aligned}
&\quad \|\nabla^{m+1} n\|_{L^3}\|\nabla^{k-m}n\|_{L^6}\|\nabla^{k+1}n\|_{L^2}\\
&\lesssim \|\nabla n\|_{L^2}^{1-\frac{m+\frac{1}{2}}{k}}\|\nabla^{k+1} n \|_{L^2}^{\frac{m+\frac{1}{2}}{k}}
         \|\nabla^\alpha n\|_{L^2}^{\frac{m+\frac{1}{2}}{k}}\|\nabla^{k+1} n\|_{L^2}^{1-\frac{m+\frac{1}{2}}{k}}
         \|\nabla^{k+1} n\|_{L^2}\\
&\lesssim \delta \|\nabla^{k+1} n\|_{L^2}^2,
\end{aligned}
\end{equation}
where $\alpha$ is defined by
\begin{equation*}
\alpha=1+\frac{k}{2m+1} \in \left(\frac{3}{2}, 2\right).
\end{equation*}
Combining \eqref{249}  with \eqref{2410}, then the term $I_{22}$ can be estimated as follow
\begin{equation}\label{2411}
I_{22}\lesssim \delta \|\nabla^{k+1} n\|_{L^2}^2.
\end{equation}
To deal with the term $I_{23}$. For the case $1\le l \le \left[\frac{k-2}{2}\right]$,
by \eqref{GN} and Holder inequality, we have
\begin{equation}\label{2412}
\begin{aligned}
&\|\nabla^{m+1} n\|_{L^6}\|\nabla^{l+1-m} n\|_{L^6}\|\nabla^{k-1-l} n\|_{L^6}\|\nabla^{k+1} n\|_{L^2}\\
&\lesssim \|\nabla^\alpha n\|_{L^2}^{1-\frac{m}{k}}\|\nabla^{k+1} n\|_{L^2}^{\frac{m}{k}}
          \|\nabla n\|_{L^2}^{1-\frac{l-m+1}{k}}\|\nabla^{k+1} n\|_{L^2}^{\frac{l-m+1}{k}}\\
&\quad    \times \|\nabla n\|_{L^2}^{\frac{l+1}{k}} \|\nabla^{k+1} n\|_{L^2}^{1-\frac{l+1}{k}}
          \|\nabla^{k+1} n\|_{L^2}\\
&\lesssim \|\nabla^\alpha n\|_{L^2}^{1-\frac{m}{k}}\|\nabla n\|_{L^2}^{1+\frac{m}{k}}
          \|\nabla^{k+1} n\|_{L^2}^2\\
&\lesssim \delta \|\nabla^{k+1} n\|_{L^2}^2,
\end{aligned}
\end{equation}
where $\alpha$ is defined by
\begin{equation*}
\alpha=1+\frac{k}{k-m}\in [2, 3).
\end{equation*}
Similarly, for the case $\left[\frac{k-2}{2}\right]+1 \le l \le k-2$, it is easy to obtain
\begin{equation}\label{2413}
\begin{aligned}
&\|\nabla^{m+1} n\|_{L^6}\|\nabla^{l+1-m} n\|_{L^6}\|\nabla^{k-1-l} n\|_{L^6}\|\nabla^{k+1} n\|_{L^2}\\
&\lesssim \|\nabla n\|_{L^2}^{1-\frac{m+1}{k}}\|\nabla^{k+1} n\|_{L^2}^{\frac{m+1}{k}}
          \|\nabla n\|_{L^2}^{1-\frac{l-m+1}{k}}\|\nabla^{k+1} n\|_{L^2}^{\frac{l-m+1}{k}}\\
&\quad    \times \|\nabla^\alpha n\|_{L^2}^{\frac{l+2}{k}} \|\nabla^{k+1} n\|_{L^2}^{1-\frac{l+2}{k}}
          \|\nabla^{k+1} n\|_{L^2}\\
&\lesssim \|\nabla n\|_{L^2}^{2-\frac{l+2}{k}} \|\nabla^\alpha n\|_{L^2}^{\frac{l+2}{k}}
          \|\nabla^{k+1} n\|_{L^2}^2\\
&\lesssim \delta \|\nabla^{k+1} n\|_{L^2}^2,
\end{aligned}
\end{equation}
where $\alpha$ is defined by
\begin{equation*}
\alpha=1+\frac{k}{l+2}\in [2,3).
\end{equation*}
Combining \eqref{2412} and \eqref{2413}, it follows that
\begin{equation}\label{2414}
I_{23} \lesssim \delta \|\nabla^{k+1} n\|_{L^2}^2.
\end{equation}
Substituting \eqref{248}, \eqref{2411} and \eqref{2414} into \eqref{247}, then we get
\begin{equation}\label{2415}
I_{2} \lesssim \delta \|\nabla^{k+1} n\|_{L^2}^2,
\end{equation}
which, together with \eqref{246}, completes the proof of lemma.
\end{proof}

Now, we turn to establish the time decay rates for the compressible nematic liquid crystal flows \eqref{2.1}-\eqref{2.4}.

\begin{lemm}\label{lemma2.5}
Under the assumptions of Theorem \ref{Decay1}, the global solution $(\varrho, u, n)$ of problem \eqref{2.1}-\eqref{2.4} satisfies
\begin{equation}\label{251}
\|\nabla^k \varrho(t)\|_{H^{N-k}}^2+\|\nabla^k u(t)\|_{H^{N-k}}^2+\|\nabla^k n(t)\|_{H^{N+1-k}}^2
\le C(1+t)^{-\frac{3}{2}-k}
\end{equation}
for $k=0,1$.
\end{lemm}
\begin{proof}
First of all, choosing the integer $k=l\in[0, N]$ in Lemma \ref{lemma2.4}, we have
\begin{equation}\label{252}
\frac{d}{dt}\left\|\nabla^l n\right\|_{L^2}^2
+C \left\|\nabla^{l+1} n \right\|_{L^2}^2
\lesssim \delta \left\|\nabla^{l+1}u\right\|_{L^2}^2.
\end{equation}
The combination of \eqref{4.7} and \eqref{252} yields
\begin{equation}\label{253}
\frac{d}{dt} \mathcal{F}^m_l(t)
+C_7\left(\|\nabla^{l+1} \varrho\|_{H^{m-l-1}}^2+\|\nabla^{l+1} u\|_{H^{m-l}}^2
         +\|\nabla^{l+1} n\|_{H^{m+1-l}}^2\right)\le 0,
\end{equation}
where $\mathcal{F}^m_l(t)$ is defined as
\begin{equation*}
\begin{aligned}
\mathcal{F}^m_l(t)=\|\nabla^l (\varrho, u)\|_{H^{m-l}}^2+\|\nabla^l n\|_{H^{m+1-l}}^2
+\frac{2C_2 \delta}{C_3} \sum_{l \le k \le m-1} \int \nabla^k u \cdot \nabla^{k+1} \varrho dx.
\end{aligned}
\end{equation*}
With the help of Young inequality, it is easy to deduce
\begin{equation}\label{254}
C_8^{-1} \!\left(\left\|\nabla^l (\varrho, u)\right\|_{H^{m-l}}^2\!+\!\left\|\nabla^{l} n \right\|_{H^{m+1-l}}^2\right)
\!\le\!
\mathcal{F}^m_l(t)
\!\le\! C_8
\left(\left\|\nabla^l (\varrho, u)\right\|_{H^{m-l}}^2\!
+\!\left\|\nabla^{l} n \right\|_{H^{m+1-l}}^2\right).
\end{equation}
Adding on both sides of \eqref{253} by $\|\nabla^l (\varrho, u, n)\|_{L^2}^2$ and applying
the equivalent relation \eqref{254}, then we have
\begin{equation}\label{255}
\frac{d}{dt}\mathcal{F}^m_l(t)+C\mathcal{F}^m_l(t)\le\|\nabla^l (\varrho, u, n)\|_{L^2}^2.
\end{equation}
Taking $l=1$ and $m=N$ specially in \eqref{255}, it arrives at
\begin{equation*}
\frac{d}{dt}\mathcal{F}^N_1(t)+C\mathcal{F}^N_1(t)\le\|\nabla (\varrho, u, n)\|_{L^2}^2,
\end{equation*}
which, together with the Gronwall inequality, gives
\begin{equation}\label{256}
\begin{aligned}
\mathcal{F}^N_1(t)
&\le \mathcal{F}^N_1(0)e^{-Ct}+\int_0^t e^{-C(t-\tau)}\|\nabla (\varrho, u, n)\|_{L^2}^2d\tau.
\end{aligned}
\end{equation}
In order to derive the time decay rate for $\mathcal{F}^N_1(t)$,  we need to control the term
$\|\nabla (\varrho, u, n)\|_{L^2}^2$. In fact, by Duhamel principle, one represents the solution
for the system \eqref{2.1}-\eqref{2.4} as
\begin{equation}\label{257}
(\varrho, u, n)^{tr}(t)=G(t)*(\varrho_0, u_0, n_0)^{tr}+\int_0^t G(t-s)*(S_1,S_2, S_3)^{tr}(s) ds.
\end{equation}
Denoting $F(t)=\underset{0 \le \tau \le t}{\sup}
(1+\tau)^{\frac{5}{2}}(\|\nabla \varrho(\tau)\|_{H^{N-1}}^2+\|\nabla u(\tau)\|_{H^{N-1}}^2
+\|\nabla n(\tau)\|_{H^{N}}^2)$, by virtue of \eqref{2.11}, \eqref{257} and Proposition
\ref{proposition2.3}, then we have
\begin{equation*}
\begin{aligned}
\|\nabla(\varrho, u, n)\|_{L^2}^2
&\le C(1+t)^{-\frac{5}{2}}
  +C\int_0^t \left(\|(S_1, S_2, S_3)\|_{L^1}^2+\|\nabla(S_1, S_2, S_3)\|_{L^2}^2\right)(1+t-\tau)^{-\frac{5}{2}}d\tau\\
&\le C(1+t)^{-\frac{5}{2}}
  +C\int_0^t \delta\left(\|\nabla \varrho\|_{H^2}^2+\|\nabla u\|_{H^2}^2
  +\|\nabla n\|_{H^2}^2 \right)(1+t-\tau)^{-\frac{5}{2}}d\tau\\
&\le C(1+t)^{-\frac{5}{2}}
  +C\delta F(t)\int_0^t (1+t-\tau)^{-\frac{5}{2}}(1+\tau)^{-\frac{5}{2}}d\tau\\
&\le C(1+t)^{-\frac{5}{2}}
  +C\delta  F(t)(1+t)^{-\frac{5}{2}},
\end{aligned}
\end{equation*}
where we have used the fact
\begin{equation*}
\begin{aligned}
&\int_0^t (1+t-\tau)^{-\frac{5}{2}}(1+\tau)^{-\frac{5}{2}}d\tau\\
&=\int_0^{\frac{t}{2}}+\int_{\frac{t}{2}}^{t}(1+t-\tau)^{-\frac{5}{2}}(1+\tau)^{-\frac{5}{2}}d\tau\\
&\le \left(1+\frac{t}{2}\right)^{-\frac{5}{2}}\int_0^{\frac{t}{2}}(1+\tau)^{-\frac{5}{2}}d\tau
+\left(1+\frac{t}{2}\right)^{-\frac{5}{2}} \int_{\frac{t}{2}}^{t}(1+t-\tau)^{-\frac{5}{2}}d\tau\\
&\le \left(1+t\right)^{-\frac{5}{2}}.
\end{aligned}
\end{equation*}
Thus, we have the estimate
\begin{equation}\label{258}
\|\nabla(\varrho, u, n)\|_{L^2}^2 \le C (1+t)^{-\frac{5}{2}}(1+\delta  F(t)).
\end{equation}
Inserting \eqref{258} into \eqref{256}, it follows immediately
\begin{equation}\label{259}
\begin{aligned}
\mathcal{F}^N_1(t)
&\le \mathcal{F}^N_1(0)e^{-Ct}+C\int_0^t e^{-C(t-\tau)} (1+\tau)^{-\frac{5}{2}}(1+\delta F(\tau))d\tau\\
&\le \mathcal{F}^N_1(0)e^{-Ct}+C(1+\delta F(t))\int_0^t e^{-C(t-\tau)} (1+\tau)^{-\frac{5}{2}}d\tau\\
&\le \mathcal{F}^N_1(0)e^{-Ct}+C(1+\delta F(t))(1+t)^{-\frac{5}{2}}\\
&\le C(1+\delta F(t))(1+t)^{-\frac{5}{2}},
\end{aligned}
\end{equation}
where we have used the simple fact
\begin{equation*}
\begin{aligned}
&\int_0^t e^{-C(t-\tau)} (1+\tau)^{-\frac{5}{2}}d\tau\\
&=\int_0^{\frac{t}{2}}+\int_{\frac{t}{2}}^t e^{-C(t-\tau)} (1+\tau)^{-\frac{5}{2}}d\tau\\
&\le e^{-\frac{c}{2}t}\int_0^{\frac{t}{2}}(1+\tau)^{-\frac{5}{2}}d\tau
     +\left(1+\frac{t}{2}\right)^{-\frac{5}{2}} \int_{\frac{t}{2}}^t e^{-C(t-\tau)}d\tau\\
&\le C\left(1+t\right)^{-\frac{5}{2}}.
\end{aligned}
\end{equation*}
Hence, by virtue of the definition of $F(t)$ and \eqref{259}, we have
\begin{equation*}
F(t) \le C(1+\delta F(t)),
\end{equation*}
which, in view of the smallness of $\delta$, gives
\begin{equation*}
F(t) \le C.
\end{equation*}
Therefore, we have the following decay rates
\begin{equation}\label{2510}
\|\nabla \varrho\|_{H^{N-1}}^2+\|\nabla u\|_{H^{N-1}}^2+\|\nabla n\|_{H^{N}}^2 \le C(1+t)^{-\frac{5}{2}}.
\end{equation}
On the other hand, by \eqref{2.11}, \eqref{257}, \eqref{2510} and Proposition \ref{proposition2.3},
it is easy to deduce
\begin{equation*}
\begin{aligned}
&\|(\varrho, u, n)(t)\|_{L^2}^2\\
&\le C(1+t)^{-\frac{3}{2}}+ C\int_0^t \left(\|(S_1, S_2, S_3)\|_{L^1}^2+\|(S_1, S_2, S_3)\|_{L^2}^2\right)(1+t-\tau)^{-\frac{3}{2}}d\tau\\
&\le C(1+t)^{-\frac{3}{2}}+C\int_0^t \delta\left(\|\nabla \varrho\|_{H^1}^2+\|\nabla u\|_{H^2}^2
  +\|\nabla n\|_{H^3}^2 \right)(1+t-\tau)^{-\frac{3}{2}}d\tau\\
&\le C(1+t)^{-\frac{3}{2}}+C\int_0^t (1+t-\tau)^{-\frac{5}{2}}(1+\tau)^{-\frac{3}{2}}d\tau\\
&\le C(1+t)^{-\frac{3}{2}},
\end{aligned}
\end{equation*}
where we have used the fact
\begin{equation*}
\int_0^t (1+t-\tau)^{-\frac{5}{2}}(1+\tau)^{-\frac{3}{2}}d\tau \le C\left(1+t\right)^{-\frac{3}{2}}.
\end{equation*}
Hence, we obtain the following time decay rates
\begin{equation*}
\|(\varrho, u, n)\|_{L^2}^2 \le C(1+t)^{-\frac{3}{2}}.
\end{equation*}
Therefore, we complete the proof of the lemma.
\end{proof}

\subsection{Improvement decay rates for the higher-order spatial derivatives of solution}

\quad In this subsection, one will improve the time decay rates for the higher-order
spatial derivatives of density, velocity and director. More precisely, we have the
following convergence rates.

\begin{lemm}\label{lemma2.6}
Under the assumptions of Theorem \ref{Decay1}, the global solution $(\varrho, u, n)$ of problem \eqref{2.1}-\eqref{2.4}
has following decay rates for all $t\ge t_0$$($$t_0$ is a constant defined below$)$,
\begin{equation}\label{261}
\|\nabla^k \varrho\|_{H^{N-k}}^2+\|\nabla^k u\|_{H^{N-k}}^2+\|\nabla^k n\|_{H^{N+1-k}}^2 \le C (1+t)^{-\frac{3}{2}-k}
\end{equation}
where $k=0, 1, 2, ..., N-1.$
\end{lemm}
\begin{proof}
We will take the strategy of induction to give the proof for the convergence rates \eqref{261}.
In fact, the inequality \eqref{251} implies \eqref{261} for the case $k=1$.
By the general step of induction, assume that the decay rates \eqref{261} hold on for the case $k=l$, i.e.
\begin{equation}\label{262}
\|\nabla^l \varrho\|_{H^{N-l}}^2+\|\nabla^l u\|_{H^{N-l}}^2+\|\nabla^l n\|_{H^{N+1-l}}^2 \le C (1+t)^{-\frac{3}{2}-l},
\end{equation}
for $l=1, 2, 3, ..., N-2.$ Then, we need to verify that \eqref{261} holds on for the case $k=l+1$.
Indeed, replacing $l$ as $l+1$ and taking $m=N$ in \eqref{253}, we have
\begin{equation*}
\frac{d}{dt} \mathcal{F}^N_{l+1}(t)
+C_7\left(\|\nabla^{l+2} \varrho\|_{H^{N-l-2}}^2+\|\nabla^{l+2} u\|_{H^{N-l-1}}^2
         +\|\nabla^{l+2} n\|_{H^{N-l}}^2\right)\le 0,
\end{equation*}
which implies
\begin{equation}\label{263}
\begin{aligned}
&\frac{d}{dt} \mathcal{F}^N_{l+1}(t)
+\frac{C_7}{2}\left(\|\nabla^{l+2} \varrho\|_{L^2}^2+\|\nabla^{l+2} \varrho\|_{H^{N-l-2}}^2
+\|\nabla^{l+2} u\|_{H^{N-l-1}}^2+\|\nabla^{l+2} n\|_{H^{N-l}}^2\right)\le 0.
\end{aligned}
\end{equation}
Denoting the time sphere $S_0$(see \cite{Schonbek1}) as follows
\begin{equation*}
S_0:=\left\{\left.\xi \in \mathbb{R}^3 \right| |\xi| \le \left(\frac{R}{1+t}\right)^{\frac{1}{2}}\right\},
\end{equation*}
where $R$ is a constant defined below. By virtue of Parseval identity, then it is easy to deduce
\begin{equation*}
\begin{aligned}
\|\nabla^{k+2} \varrho\|_{L^2}^2
&=   \int_{\mathbb{R}^3}|\xi|^{2(k+2)}|\hat{\varrho}|^2 d\xi\\
&\ge \int_{{\mathbb{R}^3}/ S_0}|\xi|^{2(k+2)}|\hat{\varrho}|^2 d\xi\\
&\ge \frac{R}{1+t}\int_{{\mathbb{R}^3}/ S_0}|\xi|^{2(k+1)}|\hat{\varrho}|^2 d\xi\\
&\ge \frac{R}{1+t}\int_{{\mathbb{R}^3}}|\xi|^{2(k+1)}|\hat{\varrho}|^2 d\xi
     -\frac{R^2}{(1+t)^2}\int_{S_0}|\xi|^{2k}|\hat{\varrho}|^2 d\xi\\
&\ge \frac{R}{1+t}\int_{{\mathbb{R}^3}}|\xi|^{2(k+1)}|\hat{\varrho}|^2 d\xi
     -\frac{R^2}{(1+t)^2}\int_{{\mathbb{R}^3}}|\xi|^{2k}|\hat{\varrho}|^2 d\xi.\\
\end{aligned}
\end{equation*}
Hence, we have the following inequality
\begin{equation}\label{264}
\|\nabla^{l+2} \varrho\|_{L^2}^2
\ge \frac{R}{1+t} \|\nabla^{l+1} \varrho\|_{L^2}^2-\frac{R^2}{(1+t)^2}\|\nabla^{l} \varrho\|_{L^2}^2.
\end{equation}
Similarly, it is easy to obtain
\begin{equation}\label{265}
\|\nabla^{k+2} u\|_{L^2}^2
\ge \frac{R}{1+t} \|\nabla^{k+1} u\|_{L^2}^2-\frac{R^2}{(1+t)^2}\|\nabla^{k} u\|_{L^2}^2.
\end{equation}
Summing up in \eqref{265} with respect to $k$ from $k=l$ to $k=N-1$, one deduces
\begin{equation}\label{266}
\|\nabla^{l+2} u\|_{H^{N-l-1}}^2
\ge \frac{R}{1+t} \|\nabla^{l+1} u\|_{H^{N-l-1}}^2-\frac{R^2}{(1+t)^2}\|\nabla^{l} u\|_{H^{N-1-l}}^2.
\end{equation}
In the same manner, it is easy to deduce
\begin{equation}\label{267}
\|\nabla^{l+2} n\|_{H^{N-l}}^2
\ge \frac{R}{1+t} \|\nabla^{l+1} n\|_{H^{N-l}}^2-\frac{R^2}{(1+t)^2}\|\nabla^{l} n\|_{H^{N-l}}^2.
\end{equation}
Substituting \eqref{264},\eqref{266} and \eqref{267} into \eqref{263}, it follows immediately
\begin{equation}\label{268}
\begin{aligned}
&\frac{d}{dt} \mathcal{F}^N_{l+1}(t)
+\frac{C_7}{2}\left[\|\nabla^{l+2} \varrho\|_{H^{N-l-2}}^2+
\frac{R}{1+t} \left(\|\nabla^{l+1} \varrho\|_{L^2}^2+\|\nabla^{l+1} u\|_{H^{N-l-1}}^2+\|\nabla^{l+1} n\|_{H^{N-l}}^2\right)
\right]\\
&\le \frac{C_7 R^2}{2(1+t)^2}\left(\|\nabla^{l} \varrho\|_{L^2}^2
+\|\nabla^{l} u\|_{H^{N-1-l}}^2+\|\nabla^{l} n\|_{H^{N-l}}^2\right).
\end{aligned}
\end{equation}
For some sufficiently large time $t \ge R-1$, we have
\begin{equation*}
\frac{R}{1+t} \le 1,
\end{equation*}
which implies
\begin{equation}\label{269}
\frac{R}{1+t} \|\nabla^{l+2} \varrho\|_{H^{N-l-2}}^2  \le \|\nabla^{l+2} \varrho\|_{H^{N-l-2}}^2.
\end{equation}
Plugging \eqref{269} into \eqref{268}, it arrives at
\begin{equation*}
\begin{aligned}
&\frac{d}{dt} \mathcal{F}^N_{l+1}(t)
+\frac{R C_7}{2(1+t)}\left(\|\nabla^{l+1} \varrho\|_{H^{N-l-1}}^2+\|\nabla^{l+1} u\|_{H^{N-l-1}}^2
+\|\nabla^{l+1} n\|_{H^{N-l}}^2\right)\\
&\le \frac{R^2 C_7}{2(1+t)^2}\left(\|\nabla^{l} \varrho\|_{L^2}^2
+\|\nabla^{l} u\|_{H^{N-1-l}}^2+\|\nabla^{l} n\|_{H^{N-l}}^2\right),
\end{aligned}
\end{equation*}
which, together with the equivalent relation \eqref{254} and the convergence rates \eqref{262}, yields
\begin{equation}\label{2610}
\frac{d}{dt} \mathcal{F}^N_{l+1}(t)+\frac{R C_7}{2C_8(1+t)}\mathcal{F}^N_{l+1}(t)\le C(1+t)^{-\frac{7}{2}-l}.
\end{equation}
Choosing
\begin{equation*}
R=\frac{2(l+3)C_8}{C_7},
\end{equation*}
and multiplying both sides of \eqref{2610} by $(1+t)^{l+3}$, we have
\begin{equation}\label{2611}
\frac{d}{dt}\left[(1+t)^{l+3}\mathcal{F}^N_{l+1}(t)\right]\le C(1+t)^{-\frac{1}{2}},
\end{equation}
for any $t\ge t_0$ and $t_0 :=\frac{2(l+3)C_8}{C_7}-1.$ Integrating \eqref{2611} over $[0, t]$,
it follows directly
\begin{equation*}
\mathcal{F}^N_{l+1}(t)\le \left[\mathcal{F}^N_{l+1}(0)+C(1+t)^{\frac{1}{2}}\right](1+t)^{-(l+3)},
\end{equation*}
which, together with equivalent relation \eqref{254}, gives rise to
\begin{equation*}
\|\nabla^{l+1} \varrho\|_{H^{N-l-1}}^2+\|\nabla^{l+1} u\|_{H^{N-l-1}}^2+\|\nabla^{l+1} n\|_{H^{N-l}}^2
\le C (1+t)^{-\frac{5}{2}-l}.
\end{equation*}
Hence, we have verified that \eqref{261} holds on for the case $k=l+1$.
By the general step of induction, we complete the proof of the lemma.
\end{proof}

Now, we will focus on improving the time decay rates for the $N-$th and $(N+1)-$th order derivatives
of direction field. More precisely, we will establish the following time decay rates.

\begin{lemm}\label{lemma2.7}
Under the assumptions of Theorem \ref{Decay1}, the global solution $(\varrho, u, n)$ of problem \eqref{2.1}-\eqref{2.4}
satisfies
\begin{equation}\label{271}
\|\nabla^k n\|_{H^{N+1-k}}^2 \le C (1+t)^{-\frac{3}{2}-k}
\end{equation}
for integer $k=0, 1, 2, ...,N+1.$
\end{lemm}
\begin{proof}
Taking $N-$th spatial derivatives to both sides of $(2.1)_3$,
multiplying $\nabla^N n$ and integrating over $\mathbb{R}^3$, then we have
\begin{equation}\label{272}
\begin{aligned}
&\frac{1}{2}\frac{d}{dt}\!\int \!|\nabla^N n|^2\! dx+\!\int \!|\nabla^{N+1} n|^2 dx\\
&=-\!\int \nabla^N (u \cdot \nabla n)\nabla^N n dx+\!\int \!\nabla^{N}(|\nabla n|^2(n+w_0))\!\nabla^N ndx.
\end{aligned}
\end{equation}
By virtue of the integration by part, Leibnitz formula, Holder and Young inequalities, it is easy to deduce
\begin{equation}\label{273}
\begin{aligned}
&-\int \nabla^N (u \cdot \nabla n)\nabla^N n \ dx\\
&\lesssim \sum_{k=0}^{N-1} \|\nabla^k u\|_{L^3}\|\nabla^{N-k}n\|_{L^6}\|\nabla^{N+1}n\|_{L^2}\\
&\lesssim \sum_{k=1}^{N-1} \|\nabla^k u\|_{H^1}\|\nabla^{N+1-k}n\|_{L^2}\|\nabla^{N+1}n\|_{L^2}
          +\|u\|_{H^1}\|\nabla^{N+1}n\|_{L^2}^2\\
&\lesssim \sum_{k=1}^{N-1} \|\nabla^k u\|_{H^1}^2\|\nabla^{N+1-k}n\|_{L^2}^2
          +(\varepsilon+\delta) \|\nabla^{N+1}n\|_{L^2}^2.\\
\end{aligned}
\end{equation}
Taking $k=N$ in \eqref{247} specially, we have
\begin{equation}\label{274}
\int \nabla^{N}(|\nabla n|^2(n+w_0))\nabla^N ndx \lesssim \delta \|\nabla^{N+1}n\|_{L^2}^2.
\end{equation}
Substituting  \eqref{273} and \eqref{274} into \eqref{272},
in view of the smallness of $\delta$ and $\varepsilon$, it arrives at
\begin{equation}\label{275}
\frac{d}{dt}\int |\nabla^N n|^2 dx+\!\int |\nabla^{N+1} n|^2 dx
\lesssim \sum_{k=1}^{N-1} \|\nabla^k u\|_{H^1}^2\|\nabla^{N+1-k}n\|_{L^2}^2.
\end{equation}
On the other hand, taking ${(N+1)}-$th spatial derivatives to both sides of $(2.1)_3$,
multiplying by $\nabla^{N+1} n$ and integrating over $\mathbb{R}^3$, then we have
\begin{equation}\label{276}
\begin{aligned}
&\frac{1}{2}\frac{d}{dt}\int |\nabla^{N+1} n|^2 dx+\!\int |\nabla^{N+2} n|^2 dx\\
&=-\!\int \nabla^{N+1} (u \cdot \nabla n)\nabla^{N+1} n dx+\!\int \nabla^{N+1}(|\nabla n|^2(n+w_0))\nabla^{N+1} n \ dx.
\end{aligned}
\end{equation}
By virtue of the Leibnitz formula, Holder and Sobolev inequalities, it is easy to deduce
\begin{equation}\label{277}
\begin{aligned}
&-\int \nabla^{N+1} (u \cdot \nabla n)\nabla^{N+1} n \ dx\\
&\lesssim \sum_{k=0}^{N-1} \|\nabla^k u\|_{L^3}\|\nabla^{N+1-k}n\|_{L^6}\|\nabla^{N+2}n\|_{L^2}
           +\|\nabla n\|_{L^\infty}\|\nabla^N u\|_{L^2}\|\nabla^{N+2} n\|_{L^2}\\
&\lesssim \sum_{k=1}^{N-1} \|\nabla^k u\|_{H^1}\|\nabla^{N+2-k}n\|_{L^2}\|\nabla^{N+2}n\|_{L^2}
           +\|u\|_{H^1}\|\nabla^{N+2}n\|_{L^2}^2\\
&\quad     +\|\nabla n\|_{L^\infty}\|\nabla^N u\|_{L^2}\|\nabla^{N+2} n\|_{L^2}\\
&\lesssim \sum_{k=1}^{N-1} \|\nabla^k u\|_{H^1}^2\|\nabla^{N+2-k}n\|_{L^2}^2
          +\|\nabla n\|_{L^\infty}^2\|\nabla^N u\|_{L^2}^2+(\varepsilon+\delta) \|\nabla^{N+2}n\|_{L^2}^2.\\
\end{aligned}
\end{equation}
On the other hand, taking $k=N+1$ in \eqref{247} specially, it is easy to deduce
\begin{equation}\label{278}
\int \nabla^{N+1}(|\nabla n|^2(n+w_0))\nabla^{N+1} n \ dx \lesssim \delta \|\nabla^{N+2}n\|_{L^2}^2.
\end{equation}
Substituting  \eqref{277} and  \eqref{278} into \eqref{276}, by virtue of the smallness of $\delta$
and $\varepsilon$, we have
\begin{equation}\label{279}
\begin{aligned}
&\frac{d}{dt}\int |\nabla^{N+1} n|^2 dx+\!\int |\nabla^{N+2} n|^2 dx\\
&\lesssim \sum_{k=1}^{N-1} \|\nabla^k u\|_{H^1}^2\|\nabla^{N+2-k}n\|_{L^2}^2
          + \|\nabla n\|_{L^\infty}^2\|\nabla^N u\|_{L^2}^2.
\end{aligned}
\end{equation}
Adding \eqref{275} to \eqref{279} and applying the Sobolev interpolation
inequality \eqref{GN}, it arrives at
\begin{equation}\label{2710}
\begin{aligned}
&\frac{d}{dt}\int(|\nabla^N n|^2+|\nabla^{N+1} n|^2)dx+\!\int(|\nabla^{N+1} n|^2+|\nabla^{N+2} n|^2)dx\\
&\lesssim \sum_{k=1}^{N-1} \|\nabla^k u\|_{H^1}^2\|\nabla^{N+1-k}n\|_{L^2}^2
          +\sum_{k=1}^{N-1} \|\nabla^k u\|_{H^1}^2\|\nabla^{N+2-k}n\|_{L^2}^2\\
&\quad \   +\|\nabla n\|_{L^\infty}^2\|\nabla^N u\|_{L^2}^2 \\
&=I\!I_{1}+I\!I_{2}+I\!I_{3}.
\end{aligned}
\end{equation}
Applying the decay rates \eqref{261}, it follows directly
\begin{equation}\label{2711}
\begin{aligned}
I\!I_{1}
&=\sum_{k=2}^{N-1} \|\nabla^k u\|_{H^1}^2\|\nabla^{N+1-k}n\|_{L^2}^2+\|\nabla u\|_{H^1}^2\|\nabla^{N}n\|_{L^2}^2\\
&\lesssim \sum_{k=2}^{N-1} (1+t)^{-(\frac{3}{2}+k)}(1+t)^{-(\frac{5}{2}+N-k)}
          +(1+t)^{-\frac{5}{2}}(1+t)^{-(\frac{1}{2}+N)}\\
&\lesssim (1+t)^{-(N+4)}+(1+t)^{-(N+3)}\\
&\lesssim (1+t)^{-(N+3)}.
\end{aligned}
\end{equation}
Similarly, we obtain
\begin{equation}\label{2712}
\begin{aligned}
I\!I_{2}
&\lesssim\sum_{k=2}^{N-1} \|\nabla^k u\|_{H^1}^2\|\nabla^{N+2-k}n\|_{L^2}^2
                  +\|\nabla u\|_{H^1}^2\|\nabla^{N+1}n\|_{L^2}^2\\
&\lesssim \sum_{k=2}^{N-1} (1+t)^{-(\frac{3}{2}+k)}(1+t)^{-(\frac{5}{2}+N-k)}
          +(1+t)^{-\frac{5}{2}}(1+t)^{-(\frac{1}{2}+N)}\\
&\lesssim (1+t)^{-(N+4)}+(1+t)^{-(N+3)}\\
&\lesssim (1+t)^{-(N+3)},\\
\end{aligned}
\end{equation}
and
\begin{equation}\label{2713}
\begin{aligned}
I\!I_{3}\lesssim (1+t)^{-\frac{5}{2}}(1+t)^{-(\frac{1}{2}+N)}\lesssim (1+t)^{-(N+3)}.
\end{aligned}
\end{equation}
Inserting  \eqref{2711}-\eqref{2713} into \eqref{2710}, it arrives at obviously
\begin{equation}\label{2714}
\frac{d}{dt}\int(|\nabla^N n|^2+|\nabla^{N+1} n|^2)dx+\int(|\nabla^{N+1} n|^2+|\nabla^{N+2} n|^2)dx
\lesssim (1+t)^{-(N+3)}.
\end{equation}
Taking the Fourier splitting method as the inequality \eqref{264} and the decay rates \eqref{261},
we have
\begin{equation}\label{2715}
\begin{aligned}
&\frac{d}{dt}\int(|\nabla^N n|^2+|\nabla^{N+1} n|^2)dx
+\frac{N+2}{1+t}\int(|\nabla^{N} n|^2+|\nabla^{N+1} n|^2)dx\\
&\lesssim \left(\frac{N+2}{1+t}\right)^2 \int(|\nabla^{N-1} n|^2+|\nabla^{N} n|^2)dx+(1+t)^{-(N+3)}\\
&\lesssim (1+t)^{-(N+\frac{5}{2})}+(1+t)^{-(N+3)}\\
&\lesssim (1+t)^{-(N+\frac{5}{2})}.
\end{aligned}
\end{equation}
Multiplying \eqref{2715} by $(1+t)^{N+2}$ and integrating the resulting inequality over $[0, t]$,
then it follows directly
\begin{equation*}
\|\nabla^N n\|_{H^1}^2 \le C(1+t)^{-\left(\frac{3}{2}+N\right)},
\end{equation*}
which, together with \eqref{261}, implies
\begin{equation}\label{2716}
\|\nabla^k n\|_{H^{N+1-k}}^2 \le C (1+t)^{-\frac{3}{2}-k}
\end{equation}
for $k=0, 1, 2, ..., N.$
On the other hand, from the inequality \eqref{279}, it arrives at
\begin{equation}\label{2717}
\begin{aligned}
&\frac{d}{dt}\int |\nabla^{N+1} n|^2 dx+\!\int |\nabla^{N+2} n|^2 dx\\
&\lesssim \sum_{k=2}^{N-1} \|\nabla^k u\|_{H^1}^2\|\nabla^{N+2-k}n\|_{L^2}^2
          +\|\nabla u\|_{H^1}^2\|\nabla^{N+1} n\|_{L^2}^2
          +\|\nabla^N u\|_{L^2}^2 \|\nabla^2 n\|_{H^1}^2\\
&\lesssim \sum_{k=2}^{N-1} (1+t)^{-(\frac{3}{2}+k)}(1+t)^{-(\frac{7}{2}+N-k)}
          +(1+t)^{-\frac{5}{2}}(1+t)^{-(\frac{3}{2}+N)}+(1+t)^{-(N+4)}\\
&\lesssim (1+t)^{-(N+5)}+(1+t)^{-(N+4)}\\
&\lesssim (1+t)^{-(N+4)}.
\end{aligned}
\end{equation}
Taking the Fourier splitting method as \eqref{264}, we have
\begin{equation}\label{2718}
\begin{aligned}
&\frac{d}{dt}\int |\nabla^{N+1} n|^2 dx+\frac{N+3}{1+t}\int |\nabla^{N+1} n|^2 dx\\
&\lesssim \left(\frac{N+3}{1+t}\right)^2 \int |\nabla^{N} n|^2 dx+(1+t)^{-(N+4)}\\
&\lesssim (1+t)^{-(\frac{7}{2}+N)}+(1+t)^{-(N+4)}\\
&\lesssim (1+t)^{-(\frac{7}{2}+N)}.
\end{aligned}
\end{equation}
Multiplying \eqref{2718} by $(1+t)^{N+3}$ and integrating the resulting inequality
over $[0, t]$, then it follows directly
\begin{equation*}
\|\nabla^{N+1} n\|_{L^2}^2 \le C(1+t)^{-\left(\frac{5}{2}+N\right)},
\end{equation*}
which, together with \eqref{2716}, completes the proof of lemma.
\end{proof}

\emph{\bf{Proof for Theorem \ref{Decay1}:}}\ With the help of Lemma \ref{lemma2.6}
and Lemma \ref{lemma2.7}, we complete the proof of Theorem \ref{Decay1}.

\section{Proof of Theorem \ref{Decay2}}

\quad In this section, we will establish the time decay rates for the mixed space-time derivatives of density, velocity
and direction field. More precisely, we have the following decay rates.

\begin{lemm}\label{lemma3.1}
Under the assumptions of Theorem \ref{Decay1}, the global solution $(\varrho, u, n)$ of problem \eqref{2.1}-\eqref{2.4}
has the time decay rates
\begin{equation}\label{3.1}
\|\nabla^k \varrho_t(t)\|_{H^{N-1-k}}^2 \le (1+t)^{-\frac{5}{2}-k},
\end{equation}
for any integer $k=0,1,...,N-2.$
\end{lemm}
\begin{proof}
Taking $k-$th (k=0,1,...,N-1) spatial derivatives to $\eqref{2.1}_1$, multiplying by $\nabla^k \varrho_t$
and integrating over $\mathbb{R}^3$, then we have
\begin{equation}\label{3.2}
\|\nabla^k \varrho_t\|_{L^2}^2
=-\int \nabla^k({\rm div}u+\varrho{\rm div}u+u\cdot \nabla \varrho)\nabla^k \varrho_tdx
=I\!I\!I_1+I\!I\!I_2+I\!I\!I_3.
\end{equation}
By virtue of the Young inequality, it follows directly
\begin{equation}\label{3.3}
I\!I\!I_1
\lesssim \|\nabla^{k+1}u\|_{L^2}^2+\varepsilon\|\nabla^k \varrho_t\|_{L^2}^2.
\end{equation}
Applying the Leibnitz formula, Young and Sobolev inequalities, we have
\begin{equation}\label{3.4}
\begin{aligned}
I\!I\!I_2
&\le \sum_{l=1}^k \|\nabla^l \varrho\|_{L^3} \|\nabla^{k+1-l}u\|_{L^6} \|\nabla^k \varrho_t\|_{L^2}
              + \|\varrho\|_{L^\infty} \|\nabla^{k+1}u\|_{L^2} \|\nabla^k \varrho_t\|_{L^2}\\
&\lesssim \sum_{l=1}^k\|\nabla^l \varrho\|_{H^1}^2 \|\nabla^{k+2-l}u\|_{L^2}^2
              + \|\varrho\|_{L^\infty}^2 \|\nabla^{k+1}u\|_{L^2}^2
              +\varepsilon\|\nabla^k \varrho_t\|_{L^2}^2\\
&\lesssim  \sum_{l=1}^k(1+t)^{-\frac{3}{2}-l}(1+t)^{-\frac{5}{2}-k+l}
              +(1+t)^{-3}(1+t)^{-\frac{3}{2}-k}
              +\varepsilon\|\nabla^k \varrho_t\|_{L^2}^2\\
&\lesssim (1+t)^{-4-k}+\varepsilon\|\nabla^k \varrho_t\|_{L^2}^2.
\end{aligned}
\end{equation}
In the same manner, it follows immediately
\begin{equation}\label{3.5}
I\!I\!I_3 \lesssim (1+t)^{-4-k}+\varepsilon\|\nabla^k \varrho_t\|_{L^2}^2.
\end{equation}
Inserting \eqref{3.3}-\eqref{3.5} into \eqref{3.2}, then it is easy to deduce
\begin{equation*}
\|\nabla^k \varrho_t\|_{L^2}^2
\lesssim (1+t)^{-4-k}+\|\nabla^{k+1}u\|_{L^2}^2
\end{equation*}
which, together with the time decay rates \eqref{1.8}, completes the proof of lemma.
\end{proof}

Next, we establish the time decay rates for the space-time derivatives of velocity.
\begin{lemm}\label{lemma3.2}
Under the assumptions of Theorem \ref{Decay1}, the global solution $(\varrho, u, n)$ of problem \eqref{2.1}-\eqref{2.4}
has the time decay rates
\begin{equation}\label{3.6}
\|\nabla^k u_t(t)\|_{L^{2}}^2 \le (1+t)^{-\frac{5}{2}-k},
\end{equation}
for any integer $k=0,1,...,N-2.$
\end{lemm}
\begin{proof}
Taking $k-$th (k=0,1,...,N-2) spatial derivatives to $\eqref{2.1}_2$, multiplying by
$\nabla^k u_t$ and integrating over $\mathbb{R}^3$, then we have
\begin{equation}\label{3.7}
\begin{aligned}
\|\nabla^k u_t\|_{L^2}^2
&=\!\int \!\nabla^k(-u \cdot \nabla u+g(\varrho)[\mu\Delta u+(\mu+\nu)\nabla {\rm div}u])\nabla^k u_tdx\\
&\quad -\!\int \!\nabla^k[(f(\varrho)+1)\nabla\varrho\!+\!g(\varrho)\nabla n\Delta n]\nabla^k u_tdx\\
&=IV_1+IV_2+IV_3+IV_4.
\end{aligned}
\end{equation}
Applying the Leibnitz formula, Holder and Sobolev inequalities, we have
\begin{equation}\label{3.8}
\begin{aligned}
IV_1
&\lesssim \sum_{l=0}^k \|\nabla^l u\|_{L^3}\|\nabla^{k+1-l}u\|_{L^6}\|\nabla^k u_t\|_{L^2}\\
&\lesssim \sum_{l=0}^k \|\nabla^l u\|_{H^1}^2\|\nabla^{k+2-l}u\|_{L^2}^2
          +\varepsilon\|\nabla^k u_t\|_{L^2}^2\\
&\lesssim \sum_{l=0}^k (1+t)^{-\frac{3}{2}-l}(1+t)^{-\frac{5}{2}-k+l}
          +\varepsilon\|\nabla^k u_t\|_{L^2}^2\\
&\lesssim (1+t)^{-4-k}+\varepsilon\|\nabla^k u_t\|_{L^2}^2.
\end{aligned}
\end{equation}
Similarly, we have
\begin{equation}\label{3.9}
\begin{aligned}
IV_2
&\lesssim \sum_{l=0}^k \|\nabla^l (g(\varrho)-1)\|_{L^\infty}\|\nabla^{k+2-l}u\|_{L^2}\|\nabla^k u_t\|_{L^2}
           +\|\nabla^{k+2}u\|_{L^2}\|\nabla^k u_t\|_{L^2}\\
&\lesssim \sum_{l=0}^k \|\nabla^l \varrho\|_{L^\infty}^2\|\nabla^{k+2-l}u\|_{L^2}^2
          +\|\nabla^{k+2}u\|_{L^2}^2
          +\varepsilon\|\nabla^k u_t\|_{L^2}^2\\
&\lesssim \sum_{l=0}^k (1+t)^{-\frac{5}{2}-l}(1+t)^{-\frac{5}{2}-k+l}
          +\|\nabla^{k+2}u\|_{L^2}^2
          +\varepsilon\|\nabla^k u_t\|_{L^2}^2\\
&\lesssim (1+t)^{-5-k}+\|\nabla^{k+2}u\|_{L^2}^2
          +\varepsilon\|\nabla^k u_t\|_{L^2}^2,
\end{aligned}
\end{equation}
and
\begin{equation}\label{3.10}
\begin{aligned}
IV_3
&\lesssim \sum_{l=0}^k \|\nabla^l [f(\varrho)+1]\|_{L^\infty}
                       \|\nabla^{k+1-l}\varrho\|_{L^2}\|\nabla^k u_t\|_{L^2}\\
&\lesssim \sum_{l=0}^k \|\nabla^l \varrho\|_{L^\infty}^2 \|\nabla^{k+1-l}\varrho\|_{L^2}^2
          +\|\nabla^{k+1}\varrho\|_{L^2}^2
          +\varepsilon\|\nabla^k u_t\|_{L^2}^2\\
&\lesssim \sum_{l=0}^k (1+t)^{-\frac{5}{2}-l}(1+t)^{-\frac{5}{2}-k+l}
          +\|\nabla^{k+1}\varrho\|_{L^2}^2
          +\varepsilon\|\nabla^k u_t\|_{L^2}^2\\
&\lesssim (1+t)^{-5-k}+\|\nabla^{k+1}\varrho\|_{L^2}^2+\varepsilon\|\nabla^k u_t\|_{L^2}^2.
\end{aligned}
\end{equation}
In the same manner, it arrives at
\begin{equation}\label{3.11}
\begin{aligned}
IV_4
&=\int \sum_{l=0}^k \sum_{m=0}^l C_k^l C_l^m \nabla^l g(\varrho)\nabla^{m+1}n\nabla^{k+2-l-m}n\nabla^k u_t dx\\
&\lesssim \sum_{l=1}^k \sum_{m=0}^l
          \|\nabla^l \varrho\|_{L^\infty}\|\nabla^{m+1}n\|_{H^1}\|\nabla^{k+3-l-m}n\|_{L^2}\|\nabla^k u_t\|_{L^2}\\
&\quad \  +\|g(\varrho)\|_{L^\infty}\|\nabla n\|_{H^1}\|\nabla^{k+3}n\|_{L^2}\|\nabla^k u_t\|_{L^2}\\
&\lesssim \sum_{l=1}^k \sum_{m=0}^l
          \|\nabla^l \varrho\|_{L^\infty}^2\|\nabla^{m+1}n\|_{H^1}^2\|\nabla^{k+3-l-m}n\|_{L^2}^2\\
&\quad \  +\|g(\varrho)\|_{L^\infty}^2\|\nabla n\|_{H^1}^2\|\nabla^{k+3}n\|_{L^2}^2
          +\varepsilon\|\nabla^k u_t\|_{L^2}^2\\
&\lesssim \sum_{l=1}^k \sum_{m=0}^l(1+t)^{-\frac{5}{2}-l}(1+t)^{-\frac{5}{2}-m}(1+t)^{-\frac{9}{2}-k+l+m}\\
&\quad \  +(1+t)^{-\frac{5}{2}}(1+t)^{-\frac{9}{2}-k}
          +\varepsilon\|\nabla^k u_t\|_{L^2}^2\\
&\lesssim (1+t)^{-7-k}+\varepsilon\|\nabla^k u_t\|_{L^2}^2.
\end{aligned}
\end{equation}
Inserting \eqref{3.8}-\eqref{3.11} into \eqref{3.7} and choosing $\varepsilon$ small enough, then we have
\begin{equation*}
\|\nabla^k u_t(t)\|_{L^2}^2
\le (1+t)^{-4-k}+\|\nabla^{k+2}u\|_{L^2}^2+\|\nabla^{k+1}\varrho\|_{L^2}^2.
\end{equation*}
Therefore, we complete the proof of lemma.
\end{proof}

Finally, we establish the time decay rates for the space-time derivatives of director.

\begin{lemm}\label{lemma3.3}
Under the assumptions of Theorem \ref{Decay1}, the global solution $(\varrho, u, n)$ of problem \eqref{2.1}-\eqref{2.4}
has the time decay rates
\begin{equation}\label{3.12}
\|\nabla^k n_t(t)\|_{L^2}^2 \le (1+t)^{-\frac{7}{2}-k},
\end{equation}
for any integer $k=0,1,...,N-1.$
\end{lemm}
\begin{proof}
Taking $k-$th (k=0,1,...,N-1) spatial derivatives to $\eqref{2.1}_3$, multiplying by $\nabla^k n_t$
and integrating over $\mathbb{R}^3$, it arrives at
\begin{equation}\label{3.13}
\|\nabla^k n_t\|_{L^2}^2
=\int \nabla^k(\Delta n-u \cdot \nabla n+|\nabla n|^2(n+w_0))\nabla^k n_tdx
=V_1+V_2+V_3.
\end{equation}
By virtue of the Young inequality, it follows immediately
\begin{equation}\label{3.14}
V_1
\lesssim \|\nabla^{k+2}n\|_{L^2}^2+\varepsilon\|\nabla^{k+1}n_t\|_{L^2}^2
\lesssim (1+t)^{-\frac{7}{2}-k}+\varepsilon\|\nabla^{k+1}n_t\|_{L^2}^2.
\end{equation}
Applying the Leibnitz formula, Holder and Young inequalities, we have
\begin{equation}\label{3.15}
\begin{aligned}
V_2
&\lesssim \sum_{l=0}^{k}\|\nabla^l u\|_{H^1}^2\|\nabla^{k+2-l}n\|_{L^2}^2+\varepsilon\|\nabla^k n_t\|_{L^2}^2\\
&\lesssim \sum_{l=0}^{k}(1+t)^{-\frac{3}{2}-l}(1+t)^{-\frac{7}{2}-k+l}+\varepsilon\|\nabla^k n_t\|_{L^2}^2\\
&\lesssim (1+t)^{-5-k}+\varepsilon\|\nabla^k n_t\|_{L^2}^2.
\end{aligned}
\end{equation}
and
\begin{equation}\label{3.16}
\begin{aligned}
V_3
&\lesssim \sum_{l=0}^{k-1}\sum_{m=0}^l
         \|\nabla^{m+1} n\|_{L^6}\|\nabla^{l+1-m}n\|_{L^6}\|\nabla^{k-l}n\|_{L^6}\|\nabla^k n_t\|_{L^2}\\
&\quad \ +\sum_{m=0}^k \|\nabla^{m+1} n\|_{L^3}\|\nabla^{l+1-m}n\|_{L^6}\|\nabla^k n_t\|_{L^2}\\
&\lesssim \sum_{l=0}^{k-1}\sum_{m=0}^l
         \|\nabla^{m+2} n\|_{L^2}^2\|\nabla^{l+2-m}n\|_{L^2}^2\|\nabla^{k+1-l}n\|_{L^2}^2\\
&\quad \ +\sum_{m=0}^k \|\nabla^{m+1} n\|_{H^1}^2\|\nabla^{l+2-m}n\|_{L^2}^2
         +\varepsilon\|\nabla^k n_t\|_{L^2}^2\\
&\lesssim \sum_{l=0}^{k-1}\sum_{m=0}^l
         (1+t)^{-\frac{7}{2}-m}(1+t)^{-\frac{7}{2}-l+m}(1+t)^{-\frac{5}{2}-k+l}\\
&\quad \ +\sum_{m=0}^k (1+t)^{-\frac{5}{2}-m}(1+t)^{-\frac{7}{2}-k+m}
         +\varepsilon\|\nabla^k n_t\|_{L^2}^2\\
&\lesssim (1+t)^{-6-k}+\varepsilon\|\nabla^k n_t\|_{L^2}^2.\\
\end{aligned}
\end{equation}
Substituting \eqref{3.14}-\eqref{3.16} into \eqref{3.13}, it arrives at
\begin{equation*}
\|\nabla^k n_t(t)\|_{L^2}^2 \le (1+t)^{-\frac{7}{2}-k}.
\end{equation*}
Therefore, we complete the proof of the lemma.
\end{proof}

\emph{\bf{Proof for Theorem \ref{Decay2}:}}\ With the help of the Lemma \ref{lemma3.1}, Lemma \ref{lemma3.2} and
Lemma \ref{lemma3.3}, we complete the proof of Theorem \ref{Decay2}.

\section{Appendix: Proof of Theorem \ref{Global-existence}}

\quad In this section, we will establish the global existence of solution for the compressible
nematic liquid crystal flows \eqref{1.1}-\eqref{1.3}. We only sketch it here since this proof is
standard. More precisely, assume there exists a constant $\delta>0$ such that
\begin{equation}\label{4.1}
\sqrt{\mathcal{E}_0^3 (t)}=\|\varrho(t)\|_{H^3}+\|u(t)\|_{H^3}+\|\nabla n(t)\|_{H^3} \le \delta.
\end{equation}
Here, we denote $\varrho=\rho-1$ and $n=d-w_0$. Then, we derive some energy estimates that play an important role for
establishing the global existence of solution under the assumption of \eqref{4.1}.
In fact, the following lemmas are easy to establish just following the idea
by Guo and Wang \cite{Guo-Wang}. Hence, we only state the results here for the sake of
brevity.

\begin{lemm}\label{lemma4.1}
If $\sqrt{\mathcal{E}_0^3(t)} \le \delta$, then we have
\begin{equation}\label{4.2}
\begin{aligned}
&\frac{d}{dt}\left\|\nabla^{k} (\varrho, u, \nabla n)\right\|_{L^2}^2
+C \left\|\nabla^{k+1} (u, \nabla n) \right\|_{L^2}^2
\lesssim \delta \left\|\nabla^{k+1}\varrho\right\|_{L^2}^2;\\
&\frac{d}{dt} \left\|\nabla^{k+1} (\varrho, u, \nabla n)\right\|_{L^2}^2
+C \left\|\nabla^{k+2} (u, \nabla n)\right\|_{L^2}^2
\lesssim \delta \left\|\nabla^{k+1} \varrho \right\|_{L^2}^2;\\
&\frac{d}{dt}\int \nabla^{k} u \cdot \nabla^{k+1} \varrho dx+C\|\nabla^{k+1} \varrho\|_{L^2}^2
\lesssim  \|\nabla^{k+1} u\|_{L^2}^2+\|\nabla^{k+2}u\|_{L^2}^2+\|\nabla^{k+3}n\|_{L^2}^2;
\end{aligned}
\end{equation}
where $k=0, 1, 2, ..., N-1$.
\end{lemm}

\emph{\bf{Proof for Theorem \ref{Global-existence}:}} Let $N \ge 3$ and
$0 \le l \le m-1$ with $3 \le m \le N.$ Summing up the estimates $\eqref{4.2}_1$
from $k=l$ to $m-1$, it is easy to deduce
\begin{equation}\label{4.3}
\frac{d}{dt}\|\nabla^l (\varrho, u, \nabla n)\|_{H^{m-l-1}}^2+C_1 \|\nabla^{l+1} (u, \nabla n)\|_{H^{m-l-1}}^2
\le C_2 \delta \|\nabla^{l+1} \varrho\|_{H^{m-l-1}}^2.
\end{equation}
Taking $k=m-1$ in $\eqref{4.2}_2$, then we get
\begin{equation}\label{4.4}
\frac{d}{dt} \|\nabla^m (\varrho, u, \nabla n)\|_{L^2}^2+C_1 \|\nabla^{m+1}(u, \nabla n)\|_{L^2}^2
\le C_2 \delta \|\nabla^m \varrho\|_{L^2}^2.
\end{equation}
Adding \eqref{4.3} to \eqref{4.4}, it is easy to obtain
\begin{equation}\label{4.5}
\frac{d}{dt}\|\nabla^l (\varrho, u, \nabla n)\|_{H^{m-l}}^2+C_1 \|\nabla^{l+1} (u, \nabla n)\|_{H^{m-l}}^2
\le C_2 \delta \|\nabla^{l+1} \varrho\|_{H^{m-l-1}}^2.
\end{equation}
Summing up the estimate $\eqref{4.2}_3$ from $k=l$ to $m-1$, we obtain
\begin{equation}\label{4.6}
\frac{d}{dt}\sum_{l \le k \le m-1} \int \nabla^k u \cdot \nabla^{k+1} \varrho dx
+C_3  \|\nabla^{l+1} \varrho\|_{H^{m-l-1}}^2
\le C_4(\|\nabla^{l+1} u\|_{H^{m-l}}^2+\|\nabla^{l+2} \nabla n\|_{H^{m-l-1}}^2).
\end{equation}
Multiplying \eqref{4.6} by $2C_2 \delta / C_3$ and adding to \eqref{4.5}, then we have
\begin{equation}\label{4.7}
\begin{aligned}
&\frac{d}{dt} \mathcal{E}^m_l(t)
+C_5 \left(\|\nabla^{l+1} \varrho\|_{H^{m-l-1}}^2+\|\nabla^{l+1} u\|_{H^{m-l}}^2
         +\|\nabla^{l+1} \nabla n\|_{H^{m-l}}^2\right)\le 0,
\end{aligned}
\end{equation}
where $\mathcal{E}^m_l(t)$ is defined as
\begin{equation*}
\begin{aligned}
\mathcal{E}^m_l(t)=\|\nabla^l (\varrho, u, \nabla n)\|_{H^{m-l}}^2
+\frac{2C_2 \delta}{C_3} \sum_{l \le k \le m-1} \int \nabla^k u \cdot \nabla^{k+1} \varrho dx.
\end{aligned}
\end{equation*}
By virtue of the Young inequality, it is easy to obtain
\begin{equation*}
\begin{aligned}
\sum_{l \le k \le m-1} \int \left| \nabla^k u \cdot \nabla^{k+1} \varrho  \right| dx
\le 2\left( \|\nabla^{l+1} \varrho\|_{H^{m-l-1}}^2+\|\nabla^l u\|_{H^{m-l-1}}^2\right).
\end{aligned}
\end{equation*}
Hence, in view of the smallness of $\delta$, we deduce
\begin{equation}\label{4.8}
\begin{aligned}
C_6^{-1} \|\nabla^l (\varrho(t), u(t), \nabla n(t))\|_{H^{m-l}}^2
\le \mathcal{E}^m_l(t)
\le C_6  \|\nabla^l (\varrho(t), u(t), \nabla n(t))\|_{H^{m-l}}^2.
\end{aligned}
\end{equation}
Integrating \eqref{4.7} over $[0, t]$, then we obtain
\begin{equation*}
\mathcal{E}^m_l(t)
+C_5 \int_0^t (\|\nabla^{l+1} \varrho\|_{H^{m-l-1}}^2+\|\nabla^{l+1} (u, \nabla n)\|_{H^{m-l}}^2) d\tau\le \mathcal{E}^m_l(0),
\end{equation*}
which, together with \eqref{4.8}, gives
\begin{equation}\label{4.9}
\begin{aligned}
&\|\nabla^l (\varrho(t), u(t),\nabla n(t))\|_{H^{m-l}}^2
+\int_0^t ( \|\nabla^{l+1} \varrho\|_{H^{m-l-1}}^2+\|\nabla^{l+1}(u, \nabla n)\|_{H^{m-l}}^2 )d\tau\\
&\le C\|\nabla^l (\varrho_0, u_0, \nabla n_0)\|_{H^{m-l}}^2.
\end{aligned}
\end{equation}
Choosing $l=0$ and $m=3$ in \eqref{4.9}, then we obtain
\begin{equation}\label{4.10}
\|\varrho(t)\|_{H^3}^2+\|u(t)\|_{H^3}^2+\|\nabla n(t)\|_{H^3}^2
\le C(\|\varrho_0\|_{H^3}^2+\|u_0\|_{H^3}^2+\|\nabla n_0\|_{H^3}^2).
\end{equation}
By a standard continuity argument, the inequality \eqref{4.10} will close the a priori estimate \eqref{4.1}.
Hence, taking $l=0$ and $m=N$ in \eqref{4.9}, it is easy to obtain
\begin{equation*}
\begin{aligned}
\|(\varrho(t), u(t),\nabla n(t))\|_{H^N}^2
\ +\int_0^t (\|\nabla \varrho\|_{H^{N-1}}^2+\|\nabla (u ,\nabla n)\|_{H^N}^2) d\tau
\le C\|(\varrho_0, u_0, \nabla n_0)\|_{H^N}^2,
\end{aligned}
\end{equation*}
which completes the proof of  Theorem \ref{Global-existence}.

\section*{Acknowledgements}
Qiang Tao's research was supported by the NSF (Grant No.11171060) and the NSF (Grant No.11301345).
Zheng-an Yao's research was supported in part by NNSFC (Grant No.11271381) and
China 973 Program (Grant No. 2011CB808002).

\phantomsection

\end{document}